\documentclass[reqno]{amsart}

\usepackage{amssymb}
\usepackage{enumitem}
\usepackage{hyperref}

\newcommand*{\mailto}[1]{\href{mailto:#1}{\nolinkurl{#1}}}
\newcommand{\arxiv}[1]{\href{http://arxiv.org/abs/#1}{arXiv: #1}}

\newtheorem*{theorem}{Theorem}
\newtheorem*{proposition}{Proposition}
\newtheorem{lemma}{Lemma}
\newtheorem*{definition}{Definition}
\newtheorem*{example}{Example}
\newtheorem*{remark}{Remark}
\newtheorem*{CP}{Coupling problem}

\newcommand{\R}{{\mathbb R}}
\newcommand{\N}{{\mathbb N}}

\newcommand{\C}{{\mathbb C}}

\newcommand{\spr}[2]{\langle #1 , #2 \rangle}

\newcommand{\E}{\mathrm{e}}
\newcommand{\I}{\mathrm{i}}

\newcommand{\tr}{\mathrm{tr}}
\newcommand{\im}{\mathrm{Im}}
\newcommand{\re}{\mathrm{Re}}

\newcommand{\OO}{\mathcal{O}}
\newcommand{\oo}{o}

\newcommand{\ledot}{\,\cdot\,}
\newcommand{\redot}{\cdot\,}

\renewcommand{\labelenumi}{(\roman{enumi})}
\renewcommand{\theenumi}{\roman{enumi}}

\begin{document}

\title{Unique solvability of a coupling problem for entire functions}

\author[J.\ Eckhardt]{Jonathan Eckhardt}
\address{Faculty of Mathematics\\ University of Vienna\\ Oskar-Morgenstern-Platz 1\\ 1090 Wien\\ Austria}
\email{\mailto{jonathan.eckhardt@univie.ac.at}}
\urladdr{\url{http://homepage.univie.ac.at/jonathan.eckhardt/}}

%\thanks{... (to appear)}
%\thanks{\href{http://dx.doi.org/...}{... (to appear)}}
%\thanks{\href{http://dx.doi.org/...}{... {\bf X} (20XX), no.~X, pp--pp}}
\thanks{{\it Research supported by the Austrian Science Fund (FWF) under Grant No.\ J3455}}

\keywords{Coupling problem for entire functions, unique solvability, inverse spectral theory}
\subjclass[2010]{Primary 30D20, 34A55; Secondary 34B05, 37K15}

\begin{abstract}
 We establish the unique solvability of a coupling problem for entire functions which arises in inverse spectral theory for singular second order ordinary differential equations/two-dimensional first order systems and is also of relevance for the integration of certain nonlinear wave equations.
\end{abstract}

\maketitle

\section*{Results}

Let $\sigma$ be a discrete set of nonzero reals such that the sum
\begin{align}\label{eqnsigTC}
 \sum_{\lambda\in\sigma} \frac{1}{|\lambda|} 
\end{align}
is finite and define the real entire function $W$ of exponential type zero by
\begin{align}\label{eqnWprod}
 W(z) = \prod_{\lambda\in\sigma} \biggl( 1-\frac{z}{\lambda} \biggr), \quad z\in\C.
\end{align}
For a given sequence $\eta\in\hat{\R}^\sigma$ (referred to as {\em coupling constants} or {\em data}), where we denote with $\hat{\R} = \R\cup\{\infty\}$ the one-point compactification of $\R$, we consider the following task. 

\begin{CP}
Find a pair of real entire functions $(\Phi_-,\Phi_+)$ of exponential type zero such that the three conditions listed below are satisfied.
\begin{enumerate}[label=\emph{(\roman*)}, leftmargin=*, widest=W]
\item[{\em (C)}] Coupling condition:\footnote{To be precise, this condition has to be read as $\Phi_+(\lambda)=0$ whenever $\eta(\lambda) = \infty$.} 
\begin{align*}
 \Phi_-(\lambda) = \eta(\lambda) \Phi_+(\lambda), \quad\lambda\in\sigma
\end{align*}
\item[{\em (G)}] Growth and positivity condition:
\begin{align*}
  \im \biggl( \frac{z \Phi_-(z) \Phi_+(z)}{W(z)}\biggr) \geq 0, \quad \im(z)>0  
\end{align*}
\item[{\em (N)}] Normalization condition: 
\begin{align*}
 \Phi_-(0) = \Phi_+(0) = 1
\end{align*}
\end{enumerate}
\end{CP}
  
Let us first assume that the pair $(\Phi_-,\Phi_+)$ is a solution of the coupling problem with data~$\eta$. 
The growth and positivity condition~(G) means that the function
\begin{align}\label{eqnGHN}
 \frac{z \Phi_-(z) \Phi_+(z)}{W(z)}, \quad z\in\C\backslash\R,
\end{align}
is a so-called Herglotz--Nevanlinna function \cite[Chapter~VI]{akgl93}, \cite{kakr74}, \cite[Chapter~5]{roro94}.
Upon invoking the open mapping theorem, this first of all guarantees that all zeros of the functions $\Phi_-$ and $\Phi_+$ are  real. 
It furthermore entails that the zeros of the function in the numerator of~\eqref{eqnGHN} and the zeros of the function in the denominator of~\eqref{eqnGHN} are interlacing  (after possible cancelations); see \cite[Theorem~27.2.1]{le96}.
From this we may conclude that the functions $\Phi_-$ and $\Phi_+$ are actually of genus zero and satisfy the bound 
\begin{align}\label{eqnPhiBound}
 |\Phi_\pm(z)| \leq \prod_{\lambda\in\sigma} \biggl( 1 + \frac{|z|}{|\lambda|} \biggr), \quad z\in\C.
\end{align}
In fact, this inequality follows essentially from roughly estimating the individual factors in the corresponding Hadamard representation, with the normalization condition~(N) taken into account, and employing the interlacing property  mentioned above.
We should emphasize here that this upper bound is always independent of the actual coupling constants $\eta$.  
On the other side, the condition~(G) also tells us that the residues of the function in~\eqref{eqnGHN} at all poles are negative. 
In conjunction with the coupling condition~(C), this implies        
\begin{align*}
   \frac{\eta(\lambda)\Phi_+(\lambda)^2}{\lambda W'(\lambda)} \leq 0 
\end{align*}
for all those $\lambda\in\sigma$ for which the coupling constant $\eta(\lambda)$ is finite. 
Unless it happens that $\lambda$ is a zero of the function $\Phi_+$,  this constitutes a necessary restriction on the sign of the coupling constant $\eta(\lambda)$ in order for a solution of the coupling problem to exist. 
Roughly speaking, the coupling constants are expected to have alternating signs beginning with non-negative ones for those corresponding to the smallest (in modulus) positive and negative element of $\sigma$. 
Motivated by these considerations and the nature of our applications, we introduce the following terminology. 

\begin{definition}
Coupling constants $\eta\in\hat{\R}^\sigma$ are called  admissible if the inequality
\begin{align*}
   \frac{\eta(\lambda)}{\lambda W'(\lambda)} \leq 0 
\end{align*}
 holds for all those $\lambda\in\sigma$ for which $\eta(\lambda)$ is finite. 
 \end{definition}
 
The main purpose of the present article is to prove that this simple condition is sufficient to guarantee unique solvability of the corresponding coupling problem. 

\begin{theorem}[Existence and Uniqueness]
 If the coupling constants $\eta\in\hat{\R}^\sigma$ are admissible, then the coupling problem with data~$\eta$ has a unique solution. 
\end{theorem}

Apart from this result, we will also establish the fact that the solution of the coupling problem depends in a continuous way on the given data. 

\begin{proposition}[Stability]
 Let $\eta_k\in\hat{\R}^\sigma$ be a sequence of admissible coupling constants that converge to some coupling constants $\eta$ (in the product topology).
 Then the solutions of the coupling problems with data~$\eta_k$ converge locally uniformly to the solution of the coupling problem with (admissible) data~$\eta$. 
\end{proposition}

In the simple case when the set $\sigma$ consists of only one point, it is possible to write down solutions explicitly in terms of the single coupling constant. 

\begin{example}
 Suppose that $\sigma = \lbrace \lambda_0 \rbrace$ for some nonzero $\lambda_0\in\R$ so that 
 \begin{align*}
   W(z) = 1-\frac{z}{\lambda_0}, \quad z\in\C. 
 \end{align*}
 From the very definition, we readily see that some $\eta\in\hat{\R}^\sigma$ is admissible if and only if the coupling constant $\eta(\lambda_0)$ is not a negative real number. 
 In this case, the unique  solution $(\Phi_-,\Phi_+)$ of the coupling problem with data~$\eta$ is given by
 \begin{align*}
  \Phi_\pm(z) = 1 - z \frac{1- \min\bigl(1,\eta(\lambda_0)^{\mp 1}\bigr)}{\lambda_0}, \quad z\in\C,
 \end{align*}
 which has to be interpreted in an appropriate way when $\eta(\lambda_0)$ is equal to zero or not finite.
 Otherwise, when the coupling constant $\eta(\lambda_0)$ is a  negative real number, the coupling problem with data~$\eta$ has no solution at all.  
 % If $(\Phi_-,\Phi_+)$ was a solution, then~\eqref{eqnetaphisign} would show that $\Phi_+(\lambda_0)=0$ and thus also $\Phi_-(\lambda_0)=0$ due to the coupling condition. Thus the function in~\eqref{eqnGHN} would have a root at zero and at $\lambda_0$ but no pole in between. 
\end{example}

The following observation sheds some light on what happens in the general situation, when the coupling constants are not necessarily admissible. 

\begin{remark}
  Let $\eta\in\hat{\R}^\sigma$ be coupling constants and define the sequence $\tilde{\eta}\in\hat{\R}^\sigma$ by 
  \begin{align*}
    \tilde{\eta}(\lambda) = \begin{cases} \eta(\lambda), & \lambda\in\sigma\backslash\rho, \\ 0, & \lambda\in\rho, \end{cases}
  \end{align*}
  where the set $\rho$ consists of all those $\lambda\in\sigma$ for which $\eta(\lambda)$ is finite and 
  \begin{align*}
    \frac{\eta(\lambda)}{\lambda W'(\lambda)} > 0.
  \end{align*}
  Since the coupling constants $\tilde{\eta}$ are admissible, there is a unique solution $(\Phi_-,\Phi_+)$ of the coupling problem with data $\tilde{\eta}$. 
 Now one can show that the coupling problem with data~$\eta$ is solvable if and only if the function $\Phi_+$ vanishes on the set $\rho$. 
 % If $(\Psi_-,\Psi_+)$ is a solution of the coupling problem with data $\eta$, then~\eqref{eqnetaphisign} shows that $\Psi_+(\lambda)=0$ and thus also $\Psi_-(\lambda)=0$ for all $\lambda\in\rho$. For this reason, the pair $(\Psi_-,\Psi_+)$ is also a solution of the coupling problem with data $\tilde{\eta}$. By the uniqueness theorem for admissible data, we see that $(\Psi_-,\Psi_+)$ coincides with $(\Phi_-,\Phi_+)$, which shows that $\Phi_+$ vanishes on $\rho$. 
 %Conversely, if $\Phi_+$ vanishes on $\rho$, then $\Phi_-(\lambda) = 0 = \eta(\lambda)\Phi_+(\lambda)$ for all $\lambda\in\rho$ and it follows readily that $(\Phi_-,\Phi_+)$ is also a solution of the coupling problem with data $\eta$. 
  In this case, the solution of the coupling problem with data $\eta$ is unique and coincides with the solution of the coupling problem with data $\tilde{\eta}$. 
\end{remark}

Before we proceed to the proofs of our results, let us point out two applications that constitute our main motivation for considering this coupling problem for entire functions. 
First and foremost, the coupling problem is essentially equivalent to an inverse spectral problem for second order ordinary differential equations or two-dimensional first order systems with trace class resolvents.  
This circumstance indicates that it is not likely for a simple elementary proof of our theorem to exist, as the uniqueness part allows one to effortlessly deduce (generalizations of) results in \cite{besasz00}, \cite{bebrwe15}, \cite{ThreeSpectra}, \cite{ConservMP}, \cite{IsospecCH}, which had to be proven in a more cumbersome way before.
On the other side, the coupling problem is also of relevance for certain completely integrable nonlinear wave equations (with the Camassa--Holm equation \cite{chh}, \cite{besasz07} and the Hunter--Saxton equation \cite{husa91} being the prime examples) when the underlying isospectral problem has purely discrete spectrum. 
For these kinds of equations, the coupling problem takes the same role as Riemann--Hilbert problems do in the case when the associated spectrum has a continuous component; see \cite{abcl91}, \cite{deitzh93}, \cite{bosh2}. 
In particular, the stability result for the coupling problem enables us to derive long-time asymptotics for solutions of such nonlinear wave equations \cite{CouplingProblem}.

%%%%%%%%%%%%%%%%%%%%%%%%%
\subsection*{Inverse spectral theory}
As a prototypical example, we are going to discuss the spectral problem for an inhomogeneous vibrating string  
\begin{align}\label{eqnDEstring}
  -f'' = z\,\omega f
\end{align}
 on the interval $(0,1)$, where $z$ is a complex spectral parameter and $\omega$ is a positive Borel measure on $(0,1)$ representing the mass distribution of the string.
 We impose a growth restriction on the measure $\omega$ to the extent that the integral
\begin{align*}
  \int_{0}^1 (1-x)x\, d\omega(x) 
\end{align*}
is finite. 
Despite both endpoints being potentially singular, these conditions guarantee that the associated Dirichlet spectrum $\sigma$ is a discrete set of positive reals such that the sum~\eqref{eqnsigTC} is finite (we refer to \cite[Section~2]{ThreeSpectra} for details). 
This fact is reflected by the existence of two solutions $\phi(z,\redot)$ and $\psi(z,\redot)$ of the differential equation~\eqref{eqnDEstring} with the asymptotics 
\begin{align*}
 \phi(z,x) & \sim x, \quad x\rightarrow 0, &  \psi(z,x) & \sim 1- x, \quad x\rightarrow 1,
\end{align*}
such that $\phi(\ledot,x)$ and $\psi(\ledot,x)$ are real entire functions of genus zero. 
Because the spectrum $\sigma$ consists precisely of those $z$ for which the solutions $\phi(z,\redot)$ and $\psi(z,\redot)$ are linearly dependent, we may infer that the function $W$ defined by~\eqref{eqnWprod} is nothing but the Wronskian of these solutions, that is, one has 
\begin{align*}
  W(z) = \psi(z,x)\phi'(z,x) - \psi'(z,x)\phi(z,x), \quad x\in(0,1),~z\in\C,
\end{align*}
where we take the unique left-continuous representatives of the derivatives. 

Our interest here lies in a  particular associated inverse spectral problem which consists in recovering the Borel measure $\omega$ from the spectrum $\sigma$ and the sequence of accompanying norming constants $\gamma_{\lambda}$ defined by 
\begin{align*}
  \gamma_{\lambda}^2 = \int_{0}^1 \phi'(\lambda,x)^2 dx, \quad \lambda\in\sigma.
\end{align*}
In order to work out the connection to the coupling problem, we first mention that for every eigenvalue $\lambda\in\sigma$ one has the relation 
\begin{align*}
  \phi(\lambda,x) = -\frac{\gamma_{\lambda}^2}{\lambda W'(\lambda)} \psi(\lambda,x), \quad x\in(0,1),
\end{align*}
which is somewhat reminiscent of the coupling condition.
Furthermore, the growth and positivity condition will be due to the fact that the function 
\begin{align*}
  \frac{z\phi(z,x)\psi(z,x)}{W(z)} = \biggl(\frac{\phi'(z,x)}{z\phi(z,x)} - \frac{\psi'(z,x)}{z\psi(z,x)}\biggr)^{-1}, \quad z\in\C\backslash\R,
\end{align*}
is  a Herglotz--Nevanlinna function for all $x\in(0,1)$.
As a final ingredient, it remains to note the identity   
\begin{align*}
 \left.\frac{\partial}{\partial z} \phi(z,x)\right|_{z=0} = -\int_0^x \int_0^s r\,d\omega(r)\, ds, \quad x\in(0,1).
\end{align*}
Upon simply normalizing the functions $\phi(\ledot,x)$ and $\psi(\ledot,x)$ at zero, we are now in the position to observe the following:
{\em For every given $x\in(0,1)$ one has  
\begin{align*}
  \int_0^x \int_0^s r\,d\omega(r)\, ds = - x\,\Phi_-'(0),
\end{align*}
where the pair $(\Phi_-,\Phi_+)$ is the unique solution of the coupling problem with (admissible) data~$\eta$ given by } 
\begin{align*}
  \eta(\lambda) = - \frac{\gamma_{\lambda}^2}{\lambda W'(\lambda)} \frac{1-x}{x}, \quad \lambda\in\sigma.
\end{align*} 
Hence we are able to retrieve the measure $\omega$ from the spectrum and the norming constants by means of solving a family of coupling problems.
In particular, this guarantees that  $\omega$ is uniquely determined by the given spectral data, a fact that usually requires considerable effort \cite{bebrwe08}, \cite{bebrwe15}, \cite{LeftDefiniteSL}, \cite{CHPencil}, \cite{ko75}. 
More generally, the coupling problem can also be employed to solve analogous inverse spectral problems for indefinite strings as in~\cite{IndefiniteString} or canonical systems with two singular endpoints.

%%%%%%%%%%%%%%%%%%%%%%%%%%%%%%
\subsection*{Nonlinear wave equations}
 Let us consider the Camassa--Holm equation 
  \begin{align*}
   u_{t} -u_{xxt}  = 2u_x u_{xx} - 3uu_x + u u_{xxx},
  \end{align*}
 which arises as a model for unidirectional wave propagation on shallow water  \cite{chh}. 
 Associated with a solution $u$ is the family of spectral problems 
 \begin{align}\label{eqnISO}
 -f'' + \frac{1}{4} f = z\, \omega(\ledot,t) f, \qquad \omega = u - u_{xx},
\end{align}
 whose significance lies in the fact that their corresponding spectra are independent of the time parameter $t$.  
 In the case when $u$ is real-valued and such that the integral 
 \begin{align*}
   \int_\R |u(x,t) - u_{xx}(x,t)|\, dx  
 \end{align*} 
 is finite for one (and hence for all) $t$, the common spectrum $\sigma$  is a discrete set of nonzero reals such that the sum~\eqref{eqnsigTC} is finite. 
 Apart from this, these assumptions also guarantee the existence of two solutions $\phi_-(z,\cdot\,,t)$ and $\phi_+(z,\cdot\,,t)$ of the differential equation~\eqref{eqnISO} with the spatial asymptotics
\begin{align*}
  \phi_-(z,x,t) & \sim \E^{\frac{x}{2}}, \quad x\rightarrow-\infty,  &  \phi_+(z,x,t) & \sim \E^{-\frac{x}{2}}, \quad x\rightarrow\infty, 
\end{align*}
such that $\phi_-(\ledot,x,t)$ and $\phi_+(\ledot,x,t)$ are real entire functions of genus zero. 
 The function $W$ defined by~\eqref{eqnWprod} is precisely the Wronskian of these solutions;
 \begin{align*}
  W(z) & = \phi_+(z,x,t) \phi_-'(z,x,t) - \phi_+'(z,x,t) \phi_-(z,x,t), \quad z\in\C,~x\in\R,
 \end{align*}
 independent of time $t$. 
 For every eigenvalue $\lambda\in\sigma$, we therefore may write 
\begin{align*}
  \phi_-(\lambda,x,t) = c_\lambda(t) \phi_+(\lambda,x,t), \quad x\in\R,
\end{align*}
 with some real-valued function $c_\lambda$.
 The crucial additional fact for this to be useful is that the time evolution for these quantities is known explicitly and given by  
\begin{align*}
 c_\lambda(t) = c_{\lambda}(0) \E^{\frac{t}{2\lambda}}, \quad \lambda\in\sigma. 
\end{align*}
Of course, this simple behavior of the spectral data is highly exceptional and only due to the completely integrable structure of the Camassa--Holm equation. 

Before we are able to substantiate the importance of the coupling problem in this context, we are left to note that the function  
\begin{align*}
  \frac{z\phi_-(z,x,t)\phi_+(z,x,t)}{W(z)} =  \biggl(\frac{\phi_-'(z,x,t)}{z\phi_-(z,x,t)} - \frac{\phi_+'(z,x,t)}{z\phi_+(z,x,t)}\biggr)^{-1}, \quad z\in\C\backslash\R,
\end{align*}
is a Herglotz--Nevanlinna function for all $x$ and $t$ with   
\begin{align*}
  \left. \frac{\partial^2}{\partial z^2} \frac{z\phi_-(z,x,t)\phi_+(z,x,t)}{W(z)}\right|_{z=0} = 4u(x,t).
\end{align*}
After taking the normalizations of the functions $\phi_-(\ledot,x,t)$ and $\phi_+(\ledot,x,t)$ at zero into account, we may now state the following: {\em For any given $x$ and $t$, we have 
\begin{align*}
  u(x,t) =     \frac{\Phi_-'(0) + \Phi_+'(0)}{2} + \frac{1}{2} \sum_{\lambda\in\sigma} \frac{1}{\lambda},
\end{align*}
where the pair $(\Phi_-,\Phi_+)$ is the unique solution of the coupling problem with (admissible) data~$\eta$ given by }
\begin{align*}
  \eta(\lambda) = c_\lambda(0)  \E^{\frac{t}{2\lambda}-x}, \quad \lambda\in\sigma. 
\end{align*}
Thus we may recover the solution $u$  by means of solving coupling problems whose data are given explicitly in terms of the associated spectral data at an initial time.

%%%%%%%%%%%%%%%%%%%%
\section*{Proofs}
%%%%%%%%%%%%%%%%%%%%

Since we are going to employ de Branges' theory of Hilbert spaces of entire functions~\cite{dB68} to establish the uniqueness part of our theorem, we begin with summarizing some necessary notation. 
First, an entire function $E$ is called a {\em de Branges function} if it satisfies the inequality 
\begin{align*}
 |E(z)| > |E(z^\ast)|
\end{align*}
for all $z$ in the open upper complex half-plane. 
Associated with such a function is a {\em de Branges space} $\mathcal{B}(E)$. 
It consists of all entire functions $F$  such that the integral 
\begin{align*}
 \int_\R \frac{|F(\lambda)|^2}{|E(\lambda)|^2} d\lambda 
\end{align*}
is finite and such that the two quotients $F/E$ and $F^\#/E$ are of bounded type in the upper half-plane with non-positive mean type, where $F^\#$ is the entire function defined by
\begin{align*}
 F^\#(z) = F(z^\ast)^\ast, \quad z\in\C.
\end{align*}
Endowed with the inner product
\begin{align*}
 \spr{F}{G} =  \int_\R \frac{F(\lambda)G(\lambda)^\ast}{|E(\lambda)|^2} d\lambda, \quad F,\,G\in\mathcal{B}(E),
\end{align*}
 the space $\mathcal{B}(E)$ turns into a reproducing kernel Hilbert space; see~\cite[Theorem~19 and Theorem~21]{dB68}.
 For each $\zeta\in\C$, the point evaluation in $\zeta$ can be written as
 \begin{align*}
  F(\zeta) = \spr{F}{K(\zeta,\cdot\,)}, \quad F\in \mathcal{B}(E),
 \end{align*}
 where the entire function $K(\zeta,\redot)$ is given by 
 \begin{align*}
  K(\zeta,z) = \frac{E(z)E^\#(\zeta^\ast)-E^\#(z) E(\zeta^\ast)}{2\pi\I (\zeta^\ast-z)}, \quad z\not=\zeta^\ast. 
 \end{align*} 
 We now show how de Branges spaces arise in connection with our coupling problem.
 
\begin{lemma}\label{lem1}
 Let $\eta\in\hat{\R}^\sigma$ be such that $\eta(\lambda)$ is finite and non-zero for every $\lambda\in\sigma$ and suppose that the pair $(\Phi_-,\Phi_+)$ is a solution of the coupling problem with data~$\eta$.
 Unless the function $\Phi_+$ is constant, there are two de Branges functions $E_1$ and $E_2$ of exponential type zero without real roots such that the following properties hold:  
\begin{enumerate}[label=(\roman*), ref=(\roman*), leftmargin=*, widest=iii]
\item The de Branges functions $E_1$ and $E_2$ are normalized by 
\begin{align*}
   -2 E_1(0) = -2 E_2(0) = 1. 
\end{align*}
\item The de Branges spaces $\mathcal{B}(E_1)$ and $\mathcal{B}(E_2)$ are both isometrically embedded in the space $L^2(\R;\mu)$, where the Borel measure $\mu$ on $\R$ is given by 
\begin{align*}
  \mu = \pi \delta_0 + \pi \sum_{\lambda\in\sigma} \frac{|\eta(\lambda)|}{|\lambda W'(\lambda)|} \delta_\lambda
\end{align*}
and $\delta_z$ denotes the unit Dirac measure centered at $z$. 
\item The corresponding reproducing kernels $K_1$ and $K_2$ satisfy the inequality
\begin{align*}
  2\pi K_2(0,0) \geq 1 \geq 2\pi K_1(0,0).
\end{align*}
\item The space $\mathcal{B}(E_1)$ is a closed subspace of $\mathcal{B}(E_2)$ with codimension at most one. 
 If $\mathcal{B}(E_1)$ coincides with $\mathcal{B}(E_2)$, then 
 \begin{align*}
   \Phi_+(z) = 2\pi K_1(0,z) = 2\pi K_2(0,z), \quad z\in\C.
 \end{align*}
 Otherwise, when $\mathcal{B}(E_1)$ has codimension one in $\mathcal{B}(E_2)$, we have 
 \begin{align*}
   \Phi_+(z) & = 2\pi K_1(0,z) + \Theta(z) \frac{1-2\pi K_1(0,0)}{\Theta(0)} \\ 
     & = 2\pi K_2(0,z) - \Theta(z) \frac{2\pi K_2(0,0)-1}{\Theta(0)}, \quad z\in\C,
 \end{align*}
 where $\Theta$ is any nontrivial function in $\mathcal{B}(E_2)$ which is orthogonal to $\mathcal{B}(E_1)$. 
\end{enumerate} 
If the function $\Phi_+$ is constant, then there is a polynomial de Branges function $E_0$ of degree one without real roots such that the following properties hold:
\begin{enumerate}[label=(\roman*), ref=(\roman*), leftmargin=*, widest=iii]
\item The de Branges function $E_0$ is normalized by 
\begin{align*}
   -2 E_0(0) = 1. 
\end{align*}
\item The de Branges space $\mathcal{B}(E_0)$ is isometrically embedded in the space $L^2(\R;\mu)$. 
\item The corresponding reproducing kernel $K_0$ satisfies the inequality
\begin{align*}
  2\pi K_0(0,0) \geq 1.
\end{align*}
\item The space $\mathcal{B}(E_0)$ is one-dimensional and  
 \begin{align*}
   \Phi_+(z) = \frac{K_0(0,z)}{K_0(0,0)}, \quad z\in\C.
 \end{align*}
\end{enumerate} 
\end{lemma}

\begin{proof}
Under the imposed conditions, all zeros of the functions $\Phi_-$ and $\Phi_+$ are simple. 
  In fact, if some $\lambda$ was a multiple zero of $\Phi_-$ or $\Phi_+$, then $\lambda$ would have to be a zero of the function $W$  as well since the function in~\eqref{eqnGHN} is a Herglotz--Nevanlinna function. 
  As this means that $\lambda$ belongs to the set $\sigma$, the coupling condition would then imply that $\lambda$ is a zero of both functions, $\Phi_-$ and $\Phi_+$, so that the function in the numerator of~\eqref{eqnGHN} would have a zero of order greater than two at $\lambda$, which constitutes a contradiction. 
  
 Let us denote with $\sigma_\pm$ the set of zeros of the entire function $\Phi_\pm$. 
 Due to the integral representation for Herglotz--Nevanlinna functions, we may write 
 \begin{align}\label{eqnGHNinv}
   -\frac{W(z)}{z\Phi_-(z)\Phi_+(z)} & = \alpha + \beta z - \frac{1}{z} +  \sum_{\lambda\in\sigma_-\cup\sigma_+} \frac{z}{\lambda(\lambda-z)} \gamma_\lambda, \quad z\in\C\backslash\R,
 \end{align}
 with some $\alpha\in\R$, $\beta\geq0$ and $\gamma_\lambda\geq0$ for every $\lambda\in\sigma_-\cup\sigma_+$  such that the sum 
 \begin{align*}
   \sum_{\lambda\in\sigma_-\cup\sigma_+} \frac{\gamma_\lambda}{\lambda^2} 
 \end{align*}
 is finite. 
 Since each $\lambda\in\sigma_-\cup\sigma_+$ is indeed a simple pole %(just note that $\Phi_-(\lambda)=W(\lambda)=0$ if and only if $\Phi_+(\lambda)=W(\lambda)=0$) 
 of the function on the left-hand side of~\eqref{eqnGHNinv}, one sees that the quantities $\gamma_\lambda$ are actually positive.  
  Now we introduce the Herglotz--Nevanlinna function $m_\pm$ by 
  \begin{align*}
    m_\pm(z) = \alpha_\pm + \beta_\pm z - \frac{1}{2z}  & +  \sum_{\lambda\in\sigma_-\cup\sigma_+} \frac{z}{\lambda(\lambda-z)} c_{\lambda,\pm} \gamma_\lambda, \quad z\in\C\backslash\R,
  \end{align*}
  where we choose $\alpha_- = \alpha$, $\alpha_+=0$, $\beta_-=\beta$, $\beta_+=0$ 
  and the quantities $c_{\lambda,\pm}\geq0$ are given by $c_{\lambda,\pm}=1$ if $\lambda\in\sigma_\pm\backslash\sigma_\mp$, $c_{\lambda,\pm}=0$ if $\lambda\in\sigma_\mp\backslash\sigma_\pm$ and 
  \begin{align*}
    c_{\lambda,\pm}^{-1} = 1+\biggl|\frac{\eta(\lambda)\Phi_+'(\lambda)}{\Phi_-'(\lambda)}\biggr|^{\pm1}
  \end{align*} 
  if $\lambda\in\sigma_-\cap\sigma_+$. 
  As a consequence of this definition, one clearly has   
  \begin{align}\label{eqnGHNinvasms}
    -\frac{W(z)}{z\Phi_-(z)\Phi_+(z)} = m_-(z) + m_+(z), \quad z\in\C\backslash\R. 
  \end{align}
  
  Since the set of nonzero poles of the function $m_\pm$ is precisely $\sigma_\pm$, we may define the real entire function $\Psi_\pm$ of exponential type zero  via 
  \begin{align*}
    \Psi_\pm(z) & = \pm z \Phi_\pm(z) m_\pm(z), \quad z\in\C\backslash\R. 
  \end{align*}
  From the identity in~\eqref{eqnGHNinvasms}, we first infer that   
    \begin{align}\label{eqnWasPP}
   \Phi_+(z) \Psi_-(z) -  \Psi_+(z) \Phi_-(z) & = W(z), \quad z\in\C,
 \end{align}
  by simply plugging in the definition of $\Psi_-$ and $\Psi_+$. 
  Moreover, one verifies that 
  \begin{align}\label{eqnPsicoup}
     \begin{aligned} |\Psi_-(\lambda)| & = |\eta(\lambda)\Psi_+(\lambda)|, \\ \Psi_-(\lambda) & = \eta(\lambda) \Psi_+(\lambda),   \end{aligned} & \begin{aligned} \quad\lambda & \in\sigma, \\  \lambda & \in\sigma\backslash(\sigma_-\cap\sigma_+), \end{aligned}
  \end{align}
  in a straightforward manner, that  the function 
  \begin{align*}
    \frac{\Psi_-(z)\Psi_+(z)}{zW(z)}, \quad z\in\C\backslash\R,
  \end{align*}
  is a Herglotz--Nevanlinna function by using~\eqref{eqnWasPP} and the  normalization
  \begin{align*}
    \Psi_\pm(0) = \lim_{z\rightarrow0} \pm z\Phi_\pm(z)m_\pm(z) = \mp \frac{1}{2}.
  \end{align*}
  
 Because the function $m_\pm$ is a non-constant Herglotz--Nevanlinna function, the entire function $E_\pm$ given by  
 \begin{align*}
  E_\pm(z) = \Psi_\pm(z) \pm z\Phi_\pm(z)\I, \quad z\in\C,
 \end{align*}
 is a de Branges function of exponential type zero. 
 Furthermore, the function $E_\pm$ does not have any real roots since otherwise the functions $\Phi_\pm$ and $\Psi_\pm$ would have a common zero which is impossible by definition. % $m_\pm$ has a genuine pole at every zero of $\Phi_\pm$ and all zeros of $\Phi_\pm$ are simple. 
 If $K_\pm$ denotes the reproducing kernel in the corresponding de Branges space $\mathcal{B}(E_\pm)$, then we have 
 \begin{align*}
   K_\pm(0,z) =\frac{\Phi_\pm(z)}{2\pi}, \quad z\in\C. 
 \end{align*}

Next, we introduce the matrix-valued Herglotz--Nevanlinna function $M$ by 
 \begin{align*}
   M(z) =   \frac{-1}{m_-(z)+m_+(z)} \begin{pmatrix} 2 &  m_-(z)-m_+(z)  \\  m_-(z)-m_+(z) &  -2m_-(z)m_+(z) \end{pmatrix}, \quad z\in\C\backslash\R.
 \end{align*}
 For such a function (see \cite[Theorem~5.4]{gets00} for example), the limit 
 \begin{align}\label{eqnOmegaDef}
   \Omega = \lim_{y\rightarrow\infty} \frac{M(\I y)}{\I y}  
 \end{align}
 exists and is a non-negative matrix. 
 Apart from this, the matrix $\Omega$ is symmetric by definition, which implies that all its entries are real. 
 Since the determinant of the matrix $M(z)$ is equal to minus one for all $z$ in the upper half-plane, we have 
 \begin{align*}
   \det{\Omega} =  \lim_{y\rightarrow\infty} \frac{\det{M(\I y)}}{-y^2}  = 0.
 \end{align*}
 Thus, we may conclude that the rank of the matrix $\Omega$ is at most one.  
 
 Let us first suppose that the matrix $\Omega$ is the null matrix, which entails that
  \begin{align*}
   \frac{\Psi_-(\I y)\Psi_+(\I y) - y^2 \Phi_-(\I y)\Phi_+(\I y)}{\I yW(\I y)} = \frac{\tr\, M(\I y)}{2} = \oo(y), \qquad y\rightarrow\infty.
 \end{align*}
 Due to the integral representation for Herglotz--Nevanlinna functions, we thus have 
 \begin{align*}
   \I\frac{E_+(z) + E_+^\#(z)Q(z)}{E_+(z) - E_+^\#(z)Q(z)} & = \frac{\Psi_-(z)\Psi_+(z) + z^2 \Phi_-(z)\Phi_+(z)}{zW(z)}  \\ 
      & = r - \frac{1}{4z} + \sum_{\lambda\in\sigma} \frac{z}{\lambda(\lambda-z)} \frac{|\Psi_-(\lambda)\Psi_+(\lambda)| + |\lambda^2 \Phi_-(\lambda)\Phi_+(\lambda)|}{|\lambda W'(\lambda)|}
     \end{align*}
 for some $r\in\R$ and all $z$ in the open upper half-plane, where $Q$ is given by 
 \begin{align*}
   Q(z) = \frac{E_-^\#(z)}{E_-(z)}.
 \end{align*}
 Upon taking the coupling condition and~\eqref{eqnPsicoup} into account, we further compute 
 \begin{align*}
   \re\frac{E_+(z) + E_+^\#(z)Q(z)}{E_+(z) - E_+^\#(z)Q(z)} = \frac{\im(z)}{4|z|^2} + \sum_{\lambda\in\sigma} \frac{\im(z)}{|\lambda-z|^2} \frac{|\eta(\lambda)|}{|\lambda W'(\lambda)|} |E_+(\lambda)|^2, \quad \im(z)>0.  
 \end{align*} 
 It now follows from \cite[Theorem~32]{dB68}, that for every function $F\in\mathcal{B}(E_+)$ one has 
 \begin{align*}
   \|F\|_{\mathcal{B}(E_+)}^2 = \pi |F(0)|^2  + \pi \sum_{\lambda\in\sigma} |F(\lambda)|^2  \frac{|\eta(\lambda)|}{|\lambda W'(\lambda)|},
 \end{align*}
 that is, the de Branges space $\mathcal{B}(E_+)$ is isometrically embedded in the space $L^2(\R;\mu)$. 
 Upon choosing $E_1=E_2=E_+$ if the function $\Phi_+$ is not a constant and $E_0=E_+$ otherwise, one readily verifies the claimed properties in this case. 
 
 If the matrix $\Omega$ has rank one, then there is a $\varphi\in[0,2\pi)$ and a $\kappa>0$ such that 
 \begin{align}\label{eqnOmegark1}
   \begin{pmatrix}  \cos\varphi & \sin\varphi \\ -\sin\varphi & \cos\varphi \end{pmatrix} \Omega  \begin{pmatrix}  \cos\varphi & -\sin\varphi \\ \sin\varphi & \cos\varphi \end{pmatrix} = \begin{pmatrix} \kappa & 0 \\ 0 & 0 \end{pmatrix}.
 \end{align}
 We now introduce the real entire functions $A_\pm$ and $B_\pm$ of exponential type zero via 
 \begin{align*}
   \begin{pmatrix} A_\pm(z) \\ B_\pm(z) \end{pmatrix} =  \begin{pmatrix}  \cos\varphi & \pm\sin\varphi \\ \mp\sin\varphi & \cos\varphi \end{pmatrix}  \begin{pmatrix} \Psi_\pm(z) \\ \pm z\Phi_\pm(z) \end{pmatrix}, \quad z\in\C.
 \end{align*}
 In view of \cite[Theorem~34]{dB68}, % \cite[Theorem~I]{dB60}
 the entire function given by 
 \begin{align*}
    A_\pm(z) + B_\pm(z)\I, \quad z\in\C,
 \end{align*}
 is a de Branges function of exponential type zero without real roots and such that the associated de Branges space coincides with $\mathcal{B}(E_\pm)$ isometrically. 
 This also guarantees that the quotient $A_\pm/B_\pm$ is a non-constant Herglotz--Nevanlinna function. 
 Furthermore, one readily sees that 
 \begin{align*}
   A_+(z) B_-(z) + B_+(z) A_-(z) = zW(z), \quad z\in\C, 
 \end{align*}
 as well as the identity 
 \begin{align*}
  \frac{2}{zW(z)} & \begin{pmatrix} -B_-(z)B_+(z) & \ast \\ \ast & A_-(z)A_+(z) \end{pmatrix} \\ & \qquad\qquad =    \begin{pmatrix}  \cos\varphi & \sin\varphi \\ -\sin\varphi & \cos\varphi \end{pmatrix} M(z)  \begin{pmatrix}  \cos\varphi & -\sin\varphi \\ \sin\varphi & \cos\varphi \end{pmatrix}, \quad z\in\C\backslash\R.
 \end{align*}
 In conjunction with~\eqref{eqnOmegaDef} and~\eqref{eqnOmegark1}, we infer that  
 \begin{align*}
  \lim_{y\rightarrow\infty}  \frac{1}{\I y} \frac{-B_-(\I y)B_+(\I y)}{\I y W(\I y)} & = \frac{\kappa}{2}, & \lim_{y\rightarrow\infty} \frac{1}{\I y} \frac{A_-(\I y)A_+(\I y)}{\I y W(\I y)} & = 0. 
 \end{align*}
 From this we may deduce that the limits % see proof in notes 
 \begin{align}\label{eqnxis}
   \xi_2 & = \lim_{y\rightarrow\infty} -\frac{1}{\I y} \frac{B_-(\I y)}{A_-(\I y)}, & \xi_1 & = \lim_{y\rightarrow\infty} - \frac{1}{\I y} \frac{B_+(\I y)}{A_+(\I y)},
 \end{align}
 exist and are positive.
 Next, we define the real entire functions $A_{j,\pm}$ and $B_{j,\pm}$ by 
 \begin{align*}
   \begin{pmatrix} A_{j,\pm}(z) \\ B_{j,\pm}(z) \end{pmatrix} = \begin{pmatrix} 1 & 0 \\ \mp (-1)^{j} \xi_j z & 1 \end{pmatrix} \begin{pmatrix} A_\pm(z) \\ B_\pm(z) \end{pmatrix}, \quad z\in\C, ~j=1,2,
 \end{align*}
 so that the functions $m_{j,\pm}$ given by 
 \begin{align*}
   m_{j,\pm}(z) = -\frac{B_{j,\pm}(z)}{A_{j,\pm}(z)}, \quad z\in\C\backslash\R,~j=1,2,
 \end{align*}
 are Herglotz--Nevanlinna functions that satisfy 
  \begin{align*}
   m_{1,-}(\I y) & \sim (\xi_1+\xi_2)\I y, & m_{1,+}(\I y) & = \oo(y), \\  m_{2,-}(\I y) & = \oo(y), & m_{2,+}(\I y) & \sim (\xi_1+\xi_2)\I y,
 \end{align*}
 as $y\rightarrow\infty$. 
 As a consequence, we may conclude that 
 \begin{align}\label{eqntratra}
   \frac{A_{j,-}(\I y)A_{j,+}(\I y) - B_{j,-}(\I y) B_{j,+}(\I y)}{\I y W(\I y)} = \oo(y), \qquad y\rightarrow\infty, ~j=1,2. 
 \end{align}
 
 In order to finish the proof, let us first suppose that the function $\Phi_+$ is not constant. 
 As then the function $m_+$ has at least two poles, we may infer that the Herglotz--Nevanlinna function $m_{1,+}$ is not constant. 
% \begin{align*}
%   m_{1,+}(z)  = - \frac{B_+(z)}{A_+(z)} - \xi_1 z = - \frac{\cos\varphi - \sin\varphi m_+(z)}{\sin\varphi + \cos\varphi m_+(z)} - \xi_1 z
% \end{align*}
% Indeed, if $m_{1,+}$ was constant, then 
% \begin{align*}
%   \cos\varphi - \sin\varphi m_+(z) = (C-\xi_1 z) (\sin\varphi + \cos\varphi m_+(z)),
% \end{align*}
% which shows that $m_+$ would have only one pole. 
 Since the same holds for $m_{2,+}$ in any case, we see that the entire functions $E_1$ and $E_2$ given by 
 \begin{align*}
   E_j(z) = A_{j,+}(z) + B_{j,+}(z)\I, \quad z\in\C,~j=1,2,
 \end{align*}
 are de Branges functions of exponential type zero without real roots. 
 Furthermore, the analytic functions $Q_1$ and $Q_2$ defined by 
 \begin{align*}
   Q_j(z) = \frac{A_{j,-}(z) - B_{j,-}(z)\I}{A_{j,-}(z) + B_{j,-}(z)\I}, \quad \im(z)>0,~j=1,2,
 \end{align*}
 are  bounded by one on the upper half-plane because the functions $m_{1,-}$ and $m_{2,-}$ are Herglotz--Nevanlinna functions.
 Due to the integral representation for Herglotz--Nevanlinna functions and~\eqref{eqntratra} we may write 
  \begin{align*}
   & \I\frac{E_{j}(z) + E_{j}^\#(z)Q_j(z)}{E_{j}(z) - E_{j}^\#(z)Q_j(z)} = \frac{A_{j,-}(z)A_{j,+}(z) + B_{j,-}(z)B_{j,+}(z)}{zW(z)}  \\ 
      & \qquad\qquad\qquad = s - \frac{1}{4z} + \sum_{\lambda\in\sigma} \frac{z}{\lambda(\lambda-z)} \frac{|A_{j,-}(\lambda)A_{j,+}(\lambda)| + |B_{j,-}(\lambda)B_{j,+}(\lambda)|}{|\lambda W'(\lambda)|}
     \end{align*}
 for some $s\in\R$, all $z$ in the open upper half-plane and $j=1,2$.
 Upon noticing that 
 \begin{align*}
   |A_{j,-}(\lambda)| & = |\eta(\lambda) A_{j,+}(\lambda)|, & |B_{j,-}(\lambda)| & = |\eta(\lambda) B_{j,+}(\lambda)|, \quad \lambda\in\sigma,~j=1,2,
 \end{align*}
 which follows from the coupling condition and~\eqref{eqnPsicoup}, we conclude that 
  \begin{align*}
   \re\frac{E_{j}(z) + E_{j}^\#(z)Q_j(z)}{E_{j}(z) - E_{j}^\#(z)Q_j(z)} = \frac{\im(z)}{4|z|^2} + \sum_{\lambda\in\sigma} \frac{\im(z)}{|\lambda-z|^2} \frac{|\eta(\lambda)|}{|\lambda W'(\lambda)|} |E_{j}(\lambda)|^2, \quad \im(z)>0.  
 \end{align*} 
  In view of \cite[Theorem~32]{dB68}, we see that the de Branges spaces $\mathcal{B}(E_1)$ and $\mathcal{B}(E_2)$ are isometrically embedded in the space $L^2(\R;\mu)$. The third item in the claim follows from the identity 
  \begin{align*}
   K_j(0,z) = K_+(0,z) - \frac{(-1)^j \xi_j\cos\varphi}{2\pi} A_+(z), \quad z\in\C,~j=1,2.
  \end{align*}
  The space $\mathcal{B}(E_1)$ is a closed subspace of $\mathcal{B}(E_2)$ with codimension at most one because (see also \cite[Theorem~33 and Theorem~34]{dB68}) we have 
  \begin{align*}
    \begin{pmatrix} A_{2,+}(z) \\ B_{2,+}(z) \end{pmatrix} = \begin{pmatrix} 1 & 0 \\ -(\xi_1+\xi_2)z & 1 \end{pmatrix} \begin{pmatrix} A_{1,+}(z) \\ B_{1,+}(z) \end{pmatrix}, \quad z\in\C, 
  \end{align*}
  and therefore the corresponding reproducing kernels are related by 
  \begin{align*}
    K_2(\zeta,z) = K_1(\zeta,z) + \frac{\xi_1+\xi_2}{\pi} A_+(z) A_+(\zeta^\ast), \quad z,\,\zeta\in\C.
  \end{align*}
  The left properties in the fourth item are readily verified upon observing that the function $A_+$ in $\mathcal{B}(E_2)$ is orthogonal to $\mathcal{B}(E_1)$ in view of  \cite[Theorem~33]{dB68} and does not vanish at zero since positivity of the second limit in~\eqref{eqnxis} would contradict the definition of $m_+$ in this case.
  It remains to note that the required normalization can be achieved by redefining $E_j$ through 
  \begin{align*}
    \begin{pmatrix} 1 & \I \end{pmatrix} \begin{pmatrix} \cos\varphi & -\sin\varphi \\ \sin\varphi & \cos\varphi \end{pmatrix} \begin{pmatrix} A_{j,+}(z) \\ B_{j,+}(z) \end{pmatrix}, \quad z\in\C,~j=1,2,
  \end{align*}
  which leaves the corresponding de Branges space unchanged \cite[Theorem~34]{dB68}. % \cite[Theorem~I]{dB60} 
  
  Otherwise, if the function $\Phi_+$ is constant, then positivity of the second limit in~\eqref{eqnxis} shows that $\sin\varphi$ is necessarily equal to zero. 
  Then the function $E_0$ given by 
  \begin{align*}
    E_0(z) = A_{2,+}(z) + B_{2,+}(z)\I, \quad z\in\C, 
%      = -\frac{1}{2} + z(\frac{\xi_2}{2}+1)\I,
  \end{align*}
  is a polynomial de Branges function % $m_{2,+}$ is non-constant
  of degree one without real roots.
  It follows as in the non-constant case above that the associated de Branges space $\mathcal{B}(E_0)$ is isometrically embedded in $L^2(\R;\mu)$. 
  Finally, observing that 
  \begin{align*}
    2\pi K_0(0,z) = 1 + \frac{\xi_2}{2}, \quad z\in\C,
  \end{align*}
  readily yields the remaining claims. 
\end{proof}

This auxiliary result in conjunction with a variant of de Branges' subspace ordering theorem \cite{ko76} allows us to verify the uniqueness part of our theorem.

\begin{proof}[Proof of uniqueness]
Let us for now suppose that the coupling constants $\eta\in\hat{\R}^\sigma$ are such that $\eta(\lambda)$ is finite and non-zero for every $\lambda\in\sigma$. 
We are going to show that any two solutions, say $(\Phi^\times_{-},\Phi^\times_{+})$ and $(\Phi^\circ_{-},\Phi^\circ_{+})$,  of the coupling problem with data~$\eta$ actually coincide. 
To this end, we first note that it suffices to verify that the functions $\Phi^\times_+$ and $\Phi^\circ_+$ are equal. 
In fact, in this case we may conclude from the integral representation for Herglotz--Nevanlinna functions that 
\begin{align*}
  \frac{z\Phi_-^\times(z)\Phi_+^\times(z)}{W(z)} = \frac{z\Phi_-^\circ(z)\Phi_+^\circ(z)}{W(z)}, \quad z\in\C\backslash\sigma,
\end{align*}
since the residues of both functions (due to the coupling condition) as well as their behavior at zero (due to the normalization) are the same, which guarantees that the functions $\Phi_-^\times$ and $\Phi_-^\circ$ coincide too. 
We distinguish the following three cases:

{\em Case 1; the functions $\Phi_+^\times$ and $\Phi_+^\circ$ are both constant.}
 The claim is obvious under these conditions since both functions are equal to one. 

{\em Case 2; precisely one of the functions $\Phi_+^\times$ and $\Phi_+^\circ$ is constant.} 
Without loss of generality, we may assume that $\Phi_+^\times$ is constant but $\Phi_+^\circ$ is not.
Let $E_0^\times$ and $E_1^\circ$ denote the corresponding de Branges functions from Lemma~\ref{lem1}. 
Since the associated de Branges spaces are both isometrically embedded in the same space $L^2(\R;\mu)$, we infer from the theorem in \cite{ko76} that either $\mathcal{B}(E_0^\times)\subseteq\mathcal{B}(E_1^\circ)$ or $\mathcal{B}(E_1^\circ)\subsetneq\mathcal{B}(E_0^\times)$. 
As the space $\mathcal{B}(E_0^\times)$ is one-dimensional, it is impossible that $\mathcal{B}(E_1^\circ)$ is a proper subspace of $\mathcal{B}(E_0^\times)$ and we conclude that $\mathcal{B}(E_0^\times)\subseteq\mathcal{B}(E_1^\circ)$. 
It follows from \cite[Theorem~33]{dB68} that there are real entire functions $\alpha$, $\beta$, $\gamma$, $\delta$ with $\alpha(0)=\delta(0) =1$ and $\beta(0)=0$ (due to the normalization of our de Branges functions) as well as
\begin{align}\label{eqnDetone}
 \alpha(z) \delta(z) - \beta(z)\gamma(z) = 1, \quad z\in\C, 
\end{align}
such that (see also \cite[Section~1]{krla14}) the quotient $\beta/\alpha$ is a Herglotz--Nevanlinna function  and  the corresponding reproducing kernels satisfy
\begin{align*}
   2\pi K_1^\circ(0,z) z = 2\pi K_0^\times(0,z) z - A_0^\times(z) \beta(z) + B_0^\times(z) (\delta(z)-1), \quad z\in\C,
\end{align*}
where $A_0^\times$ and $B_0^\times$ are real entire functions such that $E_0^\times=A_0^\times + B_0^\times \I$. 
Differentiating with respect to $z$ and evaluating at zero then gives 
\begin{align*}
 2\pi K_1^\circ(0,0) = 2\pi K_0^\times(0,0) + \frac{\beta'(0)}{2}.
\end{align*}
 Because Lemma~\ref{lem1} and the inclusion $\mathcal{B}(E_0^\times)\subseteq\mathcal{B}(E_1^\circ)$ guarantee the inequality 
\begin{align*}
 1 \leq 2\pi K_0^\times(0,0) \leq 2\pi K_1^\circ(0,0) \leq 1
\end{align*}
on the other side, we see that $\beta'(0)=0$. 
 As this means that the Herglotz--Nevanlinna function $\beta/\alpha$ has a multiple root at zero, we may conclude that $\beta$ vanishes identically. 
 Due to~\eqref{eqnDetone}, this also shows that $\delta$ has no zeros at all and thus is identically equal to one (since it is of Cartwright class \cite[Proposition~1.1]{krla14}).
 In conjunction with the remaining properties of the kernels in Lemma~\ref{lem1}, we thus get     
\begin{align*}
 \Phi_+^\times(z) = 2\pi K_0^\times(0,z) = 2\pi K_1^\circ(0,z) = \Phi_+^\circ(z), \quad z\in\C.
\end{align*}

{\em Case 3; neither of the functions $\Phi_+^\times$ and $\Phi_+^\circ$ is constant.}
Let us denote with $E_1^\times$, $E_2^\times$ and $E_1^\circ$, $E_2^\circ$ the respective corresponding de Branges functions from Lemma~\ref{lem1}. 
Since the associated de Branges spaces are all isometrically embedded in the same space $L^2(\R;\mu)$, we see from the theorem in \cite{ko76} that they are totally ordered. 
If one of the inclusions, $\mathcal{B}(E_2^\times)\subseteq\mathcal{B}(E_1^\circ)$ or $\mathcal{B}(E_2^\circ)\subseteq\mathcal{B}(E_1^\times)$, holds, then we may deduce that the functions $\Phi_+^\times$ and $\Phi_+^\circ$ are equal  by literally following the lines of the argument in the previous case. 
For this reason, it remains to verify the claim when $\mathcal{B}(E_1^\circ)\subsetneq\mathcal{B}(E_2^\times)$ and $\mathcal{B}(E_1^\times)\subsetneq\mathcal{B}(E_2^\circ)$. 
Because $\mathcal{B}(E_1^\circ)$ has codimension at most one in $\mathcal{B}(E_2^\circ)$, we see that $\mathcal{B}(E_2^\circ)\subseteq\mathcal{B}(E_2^\times)$ and analogously also $\mathcal{B}(E_2^\times)\subseteq\mathcal{B}(E_2^\circ)$, which results in  $\mathcal{B}(E_2^\times)=\mathcal{B}(E_2^\circ)$. 
After a similar argument, we furthermore see that $\mathcal{B}(E_1^\times)=\mathcal{B}(E_1^\circ)$ as well.
Now the claim follows from the properties of the corresponding reproducing kernels in Lemma~\ref{lem1}.

\smallskip

In order to prove uniqueness also under general assumptions, let $\eta\in\hat{\R}^\sigma$ be arbitrary and consider two solutions $(\Phi_-^\times,\Phi_+^\times)$ and $(\Phi_-^\circ,\Phi_+^\circ)$ of the coupling problem with data $\eta$. 
We first define the sets 
\begin{align*}
  \sigma_- & = \{\lambda\in\sigma\,|\, \eta(\lambda)=0\}, &  \sigma_+ & = \{\lambda\in\sigma\,|\, \eta(\lambda)=\infty\}, & \tilde{\sigma} & = \sigma\backslash(\sigma_+\cup\sigma_-),
\end{align*} 
 as well as the entire function $\tilde{W}$ by 
\begin{align*}
  \tilde{W}(z) = \prod_{\lambda\in\tilde{\sigma}} \biggl(1-\frac{z}{\lambda}\biggr), \quad z\in\C,
\end{align*}
and the sequence $\tilde{\eta}\in\hat{\R}^{\tilde{\sigma}}$ via 
\begin{align*}
  \tilde{\eta}(\lambda)  = \eta(\lambda) \prod_{\kappa\in\sigma_-} \biggl(1-\frac{\lambda}{\kappa}\biggr)^{-1} \prod_{\kappa\in\sigma_+} \biggl(1-\frac{\lambda}{\kappa}\biggr), \quad\lambda\in\tilde{\sigma}. 
\end{align*}
Then for any $\diamond\in\{\times,\circ\}$, the pair of real entire functions $(\tilde{\Phi}_-^\diamond,\tilde{\Phi}_+^\diamond)$ of exponential type zero defined such that  
\begin{align*}
 \tilde{\Phi}_\pm^\diamond(z) \prod_{\kappa\in\sigma_\pm} \biggl(1-\frac{z}{\kappa}\biggr)  = \Phi_\pm^\diamond(z), \quad z\in\C,
\end{align*}
satisfies first of all the coupling condition 
\begin{align*}
 \tilde{\Phi}_-^\diamond(\lambda) = \tilde{\eta}(\lambda) \tilde{\Phi}_+^\diamond(\lambda), \quad\lambda\in\tilde{\sigma}.
\end{align*}
 Furthermore, we readily see that the function 
 \begin{align*}
  \frac{z\tilde{\Phi}_-^\diamond(z)\tilde{\Phi}_+^\diamond(z)}{\tilde{W}(z)} = \frac{z\Phi_-^\diamond(z)\Phi_+^\diamond(z)}{W(z)}, \quad z\in\C\backslash\R,
\end{align*}
is a Herglotz--Nevanlinna function as well as the normalization
\begin{align*} 
\tilde{\Phi}_-^\diamond(0)=\tilde{\Phi}_+^\diamond(0)=1.
\end{align*}
In other words, the pairs $(\tilde{\Phi}_-^\times,\tilde{\Phi}_+^\times)$ and $(\tilde{\Phi}_-^\circ,\tilde{\Phi}_+^\circ)$ are solutions of the coupling problem with data $\tilde{\eta}$ when the set $\sigma$ is replaced with $\tilde{\sigma}$.
Since $\tilde{\eta}(\lambda)$ is finite and non-zero for every $\lambda\in\tilde{\sigma}$, we may invoke the first part of the proof to infer that 
\begin{align*}
  \Phi_\pm^\times(z) = \tilde{\Phi}_\pm^\times(z)  \prod_{\kappa\in\sigma_\pm} \biggl(1-\frac{z}{\kappa}\biggr) = \tilde{\Phi}_\pm^\circ(z)  \prod_{\kappa\in\sigma_\pm} \biggl(1-\frac{z}{\kappa}\biggr)   = \Phi_\pm^\circ(z), \quad z\in\C.
\end{align*}
This shows that solutions to the coupling problem are always unique. 
\end{proof}

We will require the following useful fact about rational Herglotz--Nevanlinna functions in order to establish the existence of solutions to the coupling problem. 

\begin{lemma}\label{lem2}
 If $m$ is a rational Herglotz--Nevanlinna function with a pole at zero, then there is an $N\in\N$, positive constants $l_1,\ldots,l_{N}$, real constants $\omega_1,\ldots,\omega_N$ and non-negative real constants $\upsilon_1,\ldots,\upsilon_N$ such that  
  \begin{align*}
   m(z) = \frac{p_N(z)}{q_N(z)}, \quad z\in\C\backslash\R,
 \end{align*}
 where the polynomials $p_0,\ldots,p_N$ and $q_0,\ldots,q_N$ are defined recursively via 
 \begin{align}\label{eqnpqrec}
   \begin{aligned} q_0(z) & = 0, \\ p_0(z) & = 1, \end{aligned} 
 && \begin{aligned} q_n(z) & = q_{n-1}(z) - l_n z p_{n-1}(z), \\ p_n(z) & = p_{n-1}(z) + (\omega_n+\upsilon_n z)q_n(z), \end{aligned}
  \end{align}
 for all $z\in\C$ and $n=1,\ldots,N$. 
\end{lemma}

\begin{proof}
  If the function $m$ has precisely one pole, then it admits the representation 
 \begin{align*}
   m(z) = \alpha + \beta z - \frac{1}{\gamma z}, \quad z\in\C\backslash\R, 
 \end{align*}
 for some $\alpha$, $\beta$, $\gamma\in\R$ with $\beta\geq0$ and $\gamma>0$. 
 Upon setting $N = 1$,  $\omega_1  = \alpha$,  $\upsilon_1  = \beta$ and $l_1  = \gamma$, we readily obtain the claim in this case. 
 Now let $k\in\N$, suppose that the claim holds for all functions with at most $k$ poles and assume that the function $m$ has exactly $k+1$ poles. 
 We still have    
  \begin{align*}
    m(z) = \alpha + \beta z + m_0(z), \quad z\in\C\backslash\R, 
  \end{align*}
  for some $\alpha$, $\beta\in\R$ with $\beta\geq0$ and a rational Herglotz--Nevanlinna function $m_0$ that satisfies $m_0(\I y)=\oo(1)$ as $y\rightarrow\infty$.  
  Since $m_0$ is not identically zero, we may write 
  \begin{align*}
    -\frac{1}{m_0(z)} = \gamma z + m_1(z), \quad z\in\C\backslash\R,
  \end{align*}
  for some positive constant $\gamma$ and a rational Herglotz--Nevanlinna function $m_1$ which satisfies $m_1(\I y)=\OO(1)$ as $y\rightarrow\infty$ and has less poles than $m$. 
  The function $m_1$ does not vanish identically because otherwise the function $m$ would have only one pole.  
  For this reason, the function $m_2$ defined by 
  \begin{align*}
   m_2(z) = - \frac{1}{m_1(z)}, \quad z\in\C\backslash\R,
  \end{align*}
  is a rational Herglotz--Nevanlinna function with a pole at zero but at most $k$ poles altogether.
  Due to our induction hypothesis, there is an $N\in\N$, positive constants $l_1,\ldots,l_{N}$, real constants $\omega_1,\ldots,\omega_N$ and non-negative real constants $\upsilon_1,\ldots,\upsilon_N$ such that  
  \begin{align*}
   m_2(z) = \frac{p_N(z)}{q_N(z)}, \quad z\in\C\backslash\R,
 \end{align*}
 where the polynomials $p_0,\ldots,p_N$ and $q_0,\ldots,q_N$ are given recursively by~\eqref{eqnpqrec}.  
  Upon defining the quantities $l_{N+1}=\gamma$, $\omega_{N+1}=\alpha$ and $\upsilon_{N+1}=\beta$ as well as the polynomials $p_{N+1}$ and $q_{N+1}$ via setting 
  \begin{align*}
    q_{N+1}(z) & = q_{N}(z) - l_{N+1} z p_{N}(z), & p_{N+1}(z) & = p_{N}(z) + (\omega_{N+1}+\upsilon_{N+1} z)q_{N+1}(z),
  \end{align*}
  for all $z\in\C$,   we readily compute that 
   \begin{align*}
     \frac{p_{N+1}(z)}{q_{N+1}(z)} = \omega_{N+1} + \upsilon_{N+1}z + \frac{1}{-l_{N+1}z + m_2(z)^{-1}} = m(z), \quad z\in\C\backslash\R,
   \end{align*}
   which establishes the claimed representation. 
\end{proof}

Put differently, the previous lemma says that every rational Herglotz--Nevanlinna function $m$ with a pole at zero admits a  continued fraction expansion of the form 
 \begin{align*}
   m(z) = \omega_N + \upsilon_N z + \cfrac{1}{-l_{N}z + \cfrac{1}{ \;\ddots\;  + \cfrac{1}{\omega_1 + \upsilon_1 z + \cfrac{1}{-l_{1} z}}}}, \quad z\in\C\backslash\R.
 \end{align*}
% as the polynomials there satisfy the relation
% \begin{align*}
%  \frac{p_n(z)}{q_n(z)} = \omega_n + \upsilon_n z + \frac{1}{-l_n z + \frac{q_{n-1}(z)}{p_{n-1}(z)}}, \quad z\in\C\backslash\R,~n=1,\ldots,N.
%\end{align*}
In turn, any function that can be written as such a continued fraction is a rational Herglotz--Nevanlinna function with a pole at zero. 

\begin{proof}[Proof of existence]
 Let $\eta\in\hat{\R}^\sigma$ be admissible coupling constants. 
 We will establish the existence of solutions to the coupling problem with data $\eta$ in three steps:
 
 {\em Step 1; the coupling problem with data $\eta$ is solvable when $\sigma$ is a finite set and $\eta(\lambda)$ is finite and non-zero for every $\lambda\in\sigma$.}
 Consider the function $m$ defined by 
 \begin{align*}
   m(z) = -\frac{1}{2z} - \frac{1}{2} \sum_{\lambda\in\sigma} \frac{1}{\lambda-z} \frac{\eta(\lambda)}{\lambda W'(\lambda)}, \quad z\in\C\backslash\R.
 \end{align*}
 Due to the admissibility of the coupling constants $\eta$, the function $m$ is a rational Herglotz--Nevanlinna function with a pole at zero. 
 It follows from Lemma~\ref{lem2} that there is an $N\in\N$, positive constants $l_1,\ldots,l_{N}$, real constants $\omega_1,\ldots,\omega_N$ and non-negative real constants $\upsilon_1,\ldots,\upsilon_N$ such that  
  \begin{align*}
   m(z) = \frac{p_N(z)}{q_N(z)}, \quad z\in\C\backslash\R,
 \end{align*}
 where the polynomials $p_0,\ldots,p_N$ and $q_0,\ldots,q_N$ are defined recursively via~\eqref{eqnpqrec}. 
 Because $p_N$ and $q_N$ must not have any common zeros, 
% this follows easily upon noting that 
% \begin{align*}
%   \begin{pmatrix} q_n(z) \\ p_n(z) \end{pmatrix} = \begin{pmatrix} 1 & - l_n z \\ \omega_n + \upsilon_n z & 1- l_n z(\omega_n+\upsilon_n z) \end{pmatrix} \begin{pmatrix} q_{n-1}(z) \\ p_{n-1}(z) \end{pmatrix}, \quad z\in\C,~n=1,\ldots,N,
% \end{align*}
% where the matrix is invertible
 we may conclude that 
% because of this, the zeros of $q_N$ are precisely the poles of $m$ which determines $q_N$ up to a scalar multiple (also note that all zeros are necessarily simple) 
 \begin{align}\label{eqnqNW}
   -q_N(z) = 2zW(z), \quad z\in\C,
 \end{align}
 upon also taking the residue of $m$ at zero and the fact that $p_N(0)=1$ into account. 
 Moreover, by means of evaluating the residue of $m$ at a pole $\lambda\in\sigma$, we get 
 \begin{align}\label{eqncoupatN}
   p_N(\lambda) = q_N'(\lambda)\, \mathrm{res}_\lambda m = - \eta(\lambda), \quad \lambda\in\sigma.
 \end{align}
 We now deduce from the recursion in~\eqref{eqnpqrec} that the quotient $p_n/q_n$ is a non-constant Herglotz--Nevanlinna function for all $n=1,\ldots,N$. 
 Therefore, also the function  
 \begin{align*}
     - \frac{q_n(z)}{p_{n-1}(z)} - l_n z = - \frac{q_{n-1}(z)}{p_{n-1}(z)}, \quad z\in\C\backslash\R., 
 \end{align*} 
 is a Herglotz--Nevanlinna function that is not constant if and only if $n\in\{2,\ldots,N\}$. 
 Next, we define the polynomials $r_0,\ldots,r_N$ and $s_0,\ldots,s_N$ recursively via 
 \begin{align*}
   \begin{aligned} r_N(z) & = -1, \\ s_N(z) & = 0,   \end{aligned} 
 && \begin{aligned} r_{n}(z) & = r_{n+1}(z) - (\omega_{n+1}+\upsilon_{n+1}z)s_{n+1}(z), \\ s_{n}(z) & = s_{n+1}(z) + l_{n+1} z r_{n}(z), \end{aligned}
  \end{align*}
  for all $z\in\C$ and $n=N-1,\ldots,0$. 
  One notes again that the quotient $s_n/r_{n-1}$ is a Herglotz--Nevanlinna function for all $n=N,\ldots,1$. 
  Since both sets of polynomials satisfy the same recursion, we readily compute using~\eqref{eqnqNW} that 
  \begin{align*}
    q_n(z)r_n(z) - p_n(z)s_n(z) = q_N(z)r_N(z) - p_N(z)s_N(z) = 2z W(z), \quad z\in\C,
  \end{align*}
  independent of $n=0,\ldots,N$. 
 Apart from this, we infer that for each $\lambda\in\sigma$ one has 
  \begin{align}\label{eqnpolycoup}
    p_n(\lambda) & = \eta(\lambda) r_n(\lambda), & q_n(\lambda) & = \eta(\lambda) s_n(\lambda), 
  \end{align}
  which is obvious for $n=N$ due to~\eqref{eqncoupatN} and then follows for all $n=N-1,\ldots,0$ by repeatedly employing the recursion relation. 
 % Again, it is simplest to use the recursion relation in matrix form
  Since the sum over all $l_1,\ldots,l_N$ is equal to two, 
%  \begin{align*}
%    q_n(0) & = 0, & p_n(0) & = 1, & q_n'(0) & = - \sum_{i=1}^n l_i, \quad n=0,\ldots,N, \\ 
%    s_n(0) & = 0, & r_n(0) & = -1, & s_n'(0) & = - \sum_{i=n+1}^N l_i, \quad n=0,\ldots,N,
%  \end{align*}
  we may pick an $n_0\in\{1,\ldots,N\}$ such that 
  \begin{align*}
    \sum_{i=1}^{n_0-1} l_i & \leq 1 < \sum_{i=1}^{n_0} l_i, & \delta & :=  \sum_{i=1}^{n_0} l_i -1 \in(0,l_{n_0}].
  \end{align*}
%  Define $n_0$ as the minimum
%  \begin{align*}
%    n_0 = \min\biggl\{ n\in\{1,\ldots,N\} \,|\, 1<\sum_{i=1}^n l_i \biggr\}
%  \end{align*}
%  which exists because the set if finite and contains $N$. 
  With these definitions, we introduce the real polynomials $\Phi_-$ and $\Phi_+$ such that 
  \begin{align*}
    - z \Phi_-(z) & = q_{n_0}(z) + \delta z p_{n_0-1}(z), & -z\Phi_+(z) & = s_{n_0}(z) + \delta z r_{n_0-1}(z),
  \end{align*}
  for all $z\in\C$ and first note that due to~\eqref{eqnpolycoup} we have 
  \begin{align*}
    \Phi_-(\lambda) % = \frac{q_{n_0}(\lambda)}{-\lambda} - \delta p_{n_0-1}(\lambda) = \eta(\lambda) \biggl(\frac{s_{n_0}(\lambda)}{-\lambda} - \delta r_{n_0-1}(\lambda)\biggr) 
     = \eta(\lambda) \Phi_+(\lambda), \quad \lambda\in\sigma. 
  \end{align*}
  From the considerations above, we see that the two functions 
  \begin{align*}
    - \frac{q_{n_0}(z)}{p_{n_0-1}(z)} - \delta z, && \frac{s_{n_0}(z)}{r_{n_0-1}(z)} + \delta z, \quad z\in\C\backslash\R,
  \end{align*}
  are Herglotz--Nevanlinna functions.
  Whereas the latter one is  never constant, the former one is constant if and only if $n_0=1$ and $\delta=l_1$.
  However, as this case would contradict the definition of $\delta$, we see that neither of the functions is actually constant.   
   Thus, a computation reveals that also the function  
  \begin{align*}
     -\biggl(-\frac{q_{n_0}(z)}{p_{n_0-1}(z)} - \delta z\biggr)^{-1} - \biggl(\frac{s_{n_0}(z)}{r_{n_0-1}(z)} + \delta z\biggr)^{-1} = -  \frac{2 W(z)}{z\Phi_-(z)\Phi_+(z)}, \quad z\in\C\backslash\R, 
  \end{align*}
  is a non-constant Herglotz--Nevanlinna function. 
  It remains to evaluate  
  \begin{align*}
    \Phi_-(0) & = -q_{n_0}'(0) - \delta p_{n_0-1}(0) = \sum_{i=1}^{n_0} l_i - \delta = 1, \\
    \Phi_+(0) & = -s_{n_0}'(0) - \delta r_{n_0-1}(0) = \sum_{i=n_0+1}^N l_i + \delta = \sum_{i=1}^N l_i - 1= 1, 
  \end{align*}
  to see that the pair $(\Phi_-,\Phi_+)$ is a solution of the coupling problem with data $\eta$.

 {\em Step 2; the coupling problem with data $\eta$ is solvable when $\sigma$ is a finite set.}
  Let us define the finite sets 
\begin{align*}
   \sigma_- & = \{\lambda\in\sigma\,|\, \eta(\lambda)=0\}, & \sigma_+ & = \{\lambda\in\sigma\,|\, \eta(\lambda)=\infty\}, & \tilde{\sigma} & = \sigma\backslash(\sigma_+\cup\sigma_-),
\end{align*} 
as well as the polynomial $\tilde{W}$ by 
\begin{align*}
 \tilde{W}(z) = \prod_{\lambda\in\tilde{\sigma}} \biggl(1-\frac{z}{\lambda}\biggr), \quad z\in\C,
\end{align*}
and the sequence $\tilde{\eta}\in\hat{\R}^{\tilde{\sigma}}$ via 
\begin{align*}
  \tilde{\eta}(\lambda)  = \eta(\lambda) \prod_{\kappa\in\sigma_-} \biggl(1-\frac{\lambda}{\kappa}\biggr)^{-1} \prod_{\kappa\in\sigma_+} \biggl(1-\frac{\lambda}{\kappa}\biggr), \quad\lambda\in\tilde{\sigma}. 
\end{align*}
 For every $\lambda\in\tilde{\sigma}$, the coupling constant $\tilde{\eta}(\lambda)$  is finite and non-zero with  
 \begin{align*}
   \frac{\tilde{\eta}(\lambda)}{\lambda \tilde{W}'(\lambda)} = \frac{\eta(\lambda)}{\lambda W'(\lambda)} \prod_{\kappa\in\sigma_+} \biggl(1-\frac{\lambda}{\kappa}\biggr)^2  \leq 0,
 \end{align*}
 due to the admissibility of $\eta$. 
Thus it follows from the first part of the proof that there is a pair of real entire functions $(\tilde{\Phi}_-,\tilde{\Phi}_+)$ of exponential type zero with 
\begin{align*}
  \tilde{\Phi}_-(\lambda) = \tilde{\eta}(\lambda) \tilde{\Phi}_+(\lambda), \quad \lambda\in\tilde{\sigma},
\end{align*}
such that the function 
\begin{align*}
  \frac{z\tilde{\Phi}_-(z)\tilde{\Phi}_+(z)}{\tilde{W}(z)}, \quad z\in\C\backslash\R,
\end{align*}
is a Herglotz--Nevanlinna function and such that 
\begin{align*}
  \tilde{\Phi}_-(0) = \tilde{\Phi}_+(0)  = 1. 
\end{align*}
It is now straightforward to verify that the pair $(\Phi_-,\Phi_+)$ defined by  
 \begin{align*}
   \Phi_\pm(z) = \tilde{\Phi}_\pm(z) \prod_{\lambda\in\sigma_\pm} \biggl(1-\frac{z}{\lambda}\biggr), \quad z\in\C,
\end{align*}
 is a solution of the coupling problem with data~$\eta$.

 {\em Step 3; the coupling problem with data $\eta$ is solvable.}
 For each $k\in\N$, let us define the finite set $\sigma_k=\sigma\cap[-k,k]$, the polynomial $W_k$ via 
 \begin{align*}
  W_k(z) = \prod_{\lambda\in\sigma_k} \biggl(1-\frac{z}{\lambda}\biggr), \quad z\in\C, 
 \end{align*}
 and the sequence $\eta_k\in\hat{\R}^{\sigma_k}$ by $\eta_k(\lambda) = \eta(\lambda)$ for every $\lambda\in\sigma_k$. 
 Then the inequality  
 \begin{align*}
  \frac{\eta_k(\lambda)}{\lambda W_k'(\lambda)} = \frac{\eta(\lambda)}{\lambda W'(\lambda)}  \prod_{\kappa\in\sigma\backslash\sigma_k} \biggl(1-\frac{\lambda}{\kappa}\biggr) \leq 0 
 \end{align*}
 holds for all those $\lambda\in\sigma_k$ for which $\eta_k(\lambda)$ is finite.
 More precisely, this is due to admissibility of the coupling constants $\eta$ and the fact that $|\lambda|<|\kappa|$ when $\kappa\in\sigma\backslash\sigma_k$.
 As we have seen in the second part of the proof, this guarantees that there is a pair of real entire functions $(\Phi_-^k,\Phi_+^k)$ of exponential type zero such that 
 \begin{align*}
   \Phi_-^k(\lambda) = \eta_k(\lambda) \Phi_+^k(\lambda), \quad \lambda\in\sigma_k,
 \end{align*}
 such that the function 
 \begin{align*}
  \frac{z\Phi_-^k(z)\Phi_+^k(z)}{W_k(z)}, \quad z\in\C\backslash\R, 
 \end{align*}
 is a Herglotz--Nevanlinna function and such that 
 \begin{align*}
   \Phi_-^k(0) = \Phi_+^k(0) = 1. 
 \end{align*}
 Because the estimate in~\eqref{eqnPhiBound} gives rise to the locally uniform bound 
 \begin{align*}
   |\Phi_\pm^k(z)| \leq \prod_{\lambda\in\sigma_k} \biggl(1+\frac{|z|}{|\lambda|}\biggr) \leq \prod_{\lambda\in\sigma} \biggl(1+\frac{|z|}{|\lambda|}\biggr), \quad z\in\C,
 \end{align*}
  we may choose a subsequence $k_l$ such that the pairs $(\Phi_-^{k_l},\Phi_+^{k_l})$ converge locally uniformly to a pair of real entire functions $(\Phi_-,\Phi_+)$.
 Due to the above bound, both of these functions are of exponential type zero. 
 Furthermore, it follows readily that the pair $(\Phi_-,\Phi_+)$ satisfies the coupling condition. 
 In fact, for every $\lambda\in\sigma$ we have 
 \begin{align*}
   \Phi_-^{k_l}(\lambda) = \eta(\lambda) \Phi_+^{k_l}(\lambda)
 \end{align*}
 as long as $l$ is large enough and it suffices to take the limit $l\rightarrow\infty$. 
 Of course, this has to be interpreted appropriately when $\eta(\lambda)$ is infinite.
 We are left to note that  
  \begin{align*}
    \im\biggl(\frac{z\Phi_-(z)\Phi_+(z)}{W(z)}\biggr) = \lim_{l\rightarrow\infty} \im\biggl(\frac{z\Phi^{k_l}_-(z)\Phi^{k_l}_+(z)}{W_{k_l}(z)}\biggr) \geq 0, \quad \im(z)>0,
  \end{align*}
  as well as that we have the normalization 
  \begin{align*}
    \Phi_-(0) & = \lim_{l\rightarrow\infty} \Phi^{k_l}_-(0) = 1, &  \Phi_+(0) & = \lim_{l\rightarrow\infty} \Phi^{k_l}_+(0) = 1,
  \end{align*}
  to conclude that $(\Phi_-,\Phi_+)$ is a solution of the coupling problem with data~$\eta$. 
\end{proof}

It only remains to verify that solutions depend continuously on the given data. 

\begin{proof}[Proof of stability]
  Let $\eta_k\in\hat{\R}^\sigma$ be a sequence of coupling constants that converge to some $\eta$ in the product topology and suppose that the pairs $(\Phi^k_-,\Phi^k_+)$ are solutions of the coupling problems with data $\eta_k$. 
  From the inequality in~\eqref{eqnPhiBound}, we get the locally uniform bound
   \begin{align}\label{eqnstabbound}
   |\Phi_\pm^k(z)| \leq \prod_{\lambda\in\sigma} \biggl(1+\frac{|z|}{|\lambda|}\biggr), \quad z\in\C.
 \end{align}
  If a subsequence $(\Phi^{k_l}_-,\Phi^{k_l}_+)$ converges locally uniformly to a pair $(\Phi^\infty_-,\Phi^\infty_+)$, then the functions $\Phi^\infty_-$ and $\Phi^\infty_+$ are real entire and of exponential type zero due to~\eqref{eqnstabbound}.
  When $\lambda\in\sigma$ is such that $\eta(\lambda)$ is finite, then the coupling condition yields 
  \begin{align*}
    \Phi^\infty_-(\lambda) = \lim_{l\rightarrow\infty} \Phi^{k_l}_-(\lambda) = \lim_{l\rightarrow\infty} \eta_{k_l}(\lambda) \Phi^{k_l}_+(\lambda) = \eta(\lambda) \Phi^\infty_+(\lambda). 
  \end{align*}
  In a similar manner, we see that $\Phi^\infty_+(\lambda)=0$ when $\lambda\in\sigma$ is such that $\eta(\lambda)=\infty$ is not finite. 
  Upon noting that  
  \begin{align*}
    \im\biggl(\frac{z\Phi^\infty_-(z)\Phi^\infty_+(z)}{W(z)}\biggr) = \lim_{l\rightarrow\infty} \im\biggl(\frac{z\Phi^{k_l}_-(z)\Phi^{k_l}_+(z)}{W(z)}\biggr) \geq 0, \quad \im(z)>0,
  \end{align*}
  as well as verifying the normalization 
  \begin{align*}
    \Phi^\infty_-(0) & = \lim_{l\rightarrow\infty} \Phi^{k_l}_-(0) = 1, &  \Phi^\infty_+(0) & = \lim_{l\rightarrow\infty} \Phi^{k_l}_+(0) = 1,
  \end{align*}
  we see that the pair $(\Phi^\infty_-,\Phi^\infty_+)$ is a solution of the coupling problem with data~$\eta$. 
  Since such a solution is unique, we may conclude by means of a compactness argument, using the bound~\eqref{eqnstabbound} and Montel's theorem, that the pairs $(\Phi_-^k,\Phi_+^k)$ converge locally uniformly to the unique solution of the coupling problem with data $\eta$. 
\end{proof}

%\bigskip
%\noindent
%{\bf Acknowledgments.}
% I gratefully acknowledge ...

\end{document}